\documentclass[centertags,12pt]{amsart}
\usepackage{amsfonts}
\usepackage{amsmath}
\usepackage{latexsym}
\usepackage{amssymb}

\newtheorem{thm}{Theorem}
\newtheorem{lemma}[thm]{Lemma}

\newtheorem{conj}[thm]{Conjecture}
\newtheorem{rem}[thm]{Remark}

\def\F{{\mathbb F}}
\def\L{\mathbb L}
\def\Z{\mathbb Z}
\def\Q{\mathbb Q}

\DeclareMathOperator{\Tr}{Tr}

\begin{document}

\title{On certain values of Kloosterman sums}
\author{Marko Moisio}
\email{mamo@uwasa.fi}
\address{Department of Mathematics and Statistics, University of Vaasa, P.O.
Box 700, FIN-65101 Vaasa, Finland}
\date{\today}                                 

\begin{abstract}Let $K_{q^n}(a)$ be a Kloosterman sum over the finite field $\F_{q^n}$ of characteristic $p$. In this note so called subfield conjecture is proved in case $p>3$: if $a\ne0$ belongs to the proper subfield $\F_q$ of $\F_{q^n}$, then $K_{q^n}(a)\ne-1$. This completes recent works on the subfield conjecture by Shparlinski, and Moisio and Lisonek. The problem is motivated by some applications to bent functions. Moreover, in the course of the proof a large class of translates of Dickson polynomials are shown to be irreducible.   
\end{abstract}

\maketitle


\section{Introduction}

Let $p$ be a prime number and let $q>1$ be a power of $p$. For a positive integer $n$, let $\F_{q^n}$ denote the finite field of order $q^n$ and let $b\in\F_{q^n}^*$. Let $\Tr:\F_{q^n}\rightarrow\F_p$ be the trace function from $\F_{q^n}$ onto $\F_p$.
The Kloosterman sum $K_{q^n}(b)$ on $\F_{q^n}^*$ is defined by
\begin{equation}\label{e:klsum}
K_{q^n}(b)=\sum_{x\in\F_{q^n}^*}\zeta_p^{\Tr(x+bx^{-1})},
\end{equation}
where $\zeta_p=e^{2\pi i/p}$ is a primitive $p$th root of unity.

\begin{rem}Sometimes Kloosterman sums are defined over $\F_{q^n}$ by setting $0^{-1}=0$ and denoted by $\mathcal K_{q^n}(b)$. Then 
$\mathcal K_{q^n}(b):=K_{q^n}(b)+1$. Throughout this note the classical definition~\eqref{e:klsum} of a Kloosterman sums is used. 
\end{rem}

The following conjecture, which is related to the problem of constructing Dillon type bent functions (see~\cite{CG,HK,L,LM}), is implicitly contained in~\cite{L}. 
\begin{conj}
\label{subfield-conjecture}
Assume $n>1$. If $a\in\F_{q}^*$ and $K_{q^n}(a)=-1$, then $q^n=16$.
\end{conj}

In a recent paper by Shparlinski~\cite{S} the following result was obtained.
\begin{thm}[{\rm \cite{S}}]\label{Igor-thm} Let $a\in\F_q^*$, and let $n>s_0(p)$ where
\\
\noindent $s_0(p)$:=
\vskip-13pt
\[
\begin{cases}
15& \hskip-8pt\mbox{ if } p=2,3, \\
\max\{ 2^{p-1}-1, 2000((p-1)\log(3(p-1)))^{12} \} &\hskip-8pt \mbox{ if }p>3.
\end{cases}
\]
Then $K_{q^n}(a)\neq-1$.
\end{thm}
 
Moisio and Lisonek~\cite{LM} showed that Conjecture~\ref{subfield-conjecture} is true also in the remaining cases when $p=2,3$ i.e. for $1<n\le15$. The case $(p,n)=(2,2)$ was first proved in~\cite{CG}. 

Actually, by the proof of Theorem~\ref{Igor-thm} in~\cite{S} it is seen that $s_0(p)=15$ if $K_q(a)\in\Z$, and this observation was used to prove the following

\begin{thm}[{\rm \cite{LM}}]\label{t:int} Let $n>1$ and assume $K_q(a)\in\Z$. If $q^n\ne16$, then $K_{q^n}(a)\ne-1$.
\end{thm}

The aim of this paper is to prove the open cases of the Conjecture~\ref{subfield-conjecture} i.e. we shall prove the following

\begin{thm}\label{t:mainth}Let $p>3$, and let $a\in\F_{q}^*$. Then $K_{q^n}(a)\ne-1$.
\end{thm} 
 
The idea of the proof is based on the observation that if $K_{q^n}(a)=-1$ for some $a\in\F_{q}^*$, then the minimal polynomial of $K_q(a)$ over $\Z$ must be a factor of $D_n(x,q)+1$ where $D_n(x,q)$ is a Dickson polynomial. To show that this is impossible, we prove Theorem~\ref{t:irred_dickson} below, which provides us a large class of irreducible (translates of Dickson) polynomials and in particular, shows that $D_n(x,q)+1$ is irreducible if $n$ is odd. The proof uses a reducibility result on translates of Dickson polynomials by Turnwald~\cite{T} together with the deep result on the primitive divisors of Lucas numbers by Bilu, Hanrot, and Voutier~\cite{BHV}. 

We end this section with a simple observation:

\begin{rem}\label{r:reduction} Let $\ell$ be a prime. Assume $K_{q^\ell}(a)\ne-1$, where $\F_q$ is any finite field containing $a$. Then $K_{q^{s\ell}}(a)\ne-1$ for all positive integers $s$.
\end{rem}

\section{Proof of Theorem~\ref{t:mainth}}

Let $a$ be a fixed element in $\F_q^*$, and assume that $p>3$ and $n>1$. The following result by Carlitz is a key (see~\cite{C} or~\cite[Theorem 5.46]{LN}). 
 
\begin{thm}[{\rm \cite{C}}]\label{t:Carlitz}Let
\[
	D_n(x,q)=\sum_{i=0}^{\lfloor n/2\rfloor}\frac{n}{n-i}\binom{n-i}i(-q)^ix^{n-2i}\in\Z[x]
\]
be the Dickson polynomial of the first kind of degree $n$ with parameter $q$. Then,
\[
	K_{q^n}(a) =(-1)^{n-1}D_n(K_q(a),q).
\]
\end{thm}

The next two theorems are used to show that $D_n(x,q)+1$ is irreducible if $n$ is any odd integer. The first one is actually a special case of the Theorem in~\cite{T}.
  
\begin{thm}[{\rm \cite{T}}]\label{t:Tu} Let $r,n$ be integers with $n>0$ odd. Then, $D_n(x,r)+1$ is reducible in $\Q[x]$ if and only if there is a prime factor $\ell$ of $n$ such that $D_\ell(c,r^{n/\ell})=-1$ for some $c\in\Q$. 
\end{thm}

We recall a deep result from number theory on the primitive divisors of Lucas numbers, used also in~\cite{S}. Let $\L$ be an algebraic number field. Given two algebraic integers $\alpha,\beta\in\L$ and a positive integer $k$ we say that $\alpha^k-\beta^k$ has a {\it primitive divisor} if there is a prime ideal $\mathfrak p$ in the ring of integers of $\L$ which divides $\alpha^k-\beta^k$ but does not divide $\alpha^j-\beta^j$ for $j=1\dots,k-1$.

\begin{thm}[{\rm \cite{BHV}}]\label{t:BHV} Let $\alpha,\beta$ be algebraic integers such that $\alpha+\beta$ and $\alpha\beta$ are nonzero coprime rational integers and $\alpha/\beta$ is not a root of unity. Then $\alpha^k-\beta^k$ has a primitive divisor for every integer $k>30$.
\end{thm}

\begin{thm}\label{t:irred_dickson}Let $n,r\in\Z$ with $r\ne0,\pm1$ and $n>0$ odd. Then $D_n(x,r)+1$ is irreducible over $\Q$.
\end{thm}

\begin{proof}We may assume that $n>1$. Let $\ell$ be a prime factor of $n$, and let $s=r^{n/\ell}$. By Theorem~\ref{t:Tu}, it is enough to show that $D_{\ell}(c,s)\ne-1$ for all $c\in\Q$. Assume $D_\ell(c,s)=-1$ for some $c\in\Q$. Since the constant coefficient of $D_\ell(x,s)+1$ is 1, we get $c=\pm1$. On the other hand, it is well known that $D_\ell(c,s)=\alpha^\ell+\beta^\ell$, where $\alpha=\frac12(c+\sqrt{c^2-4s})$ and $\beta=\frac12(c-\sqrt{c^2-4s})$. Now $\alpha+\beta,\alpha\beta\in\Z$, and by the similar reasoning as in~\cite{S} we see that $\rho:=\alpha/\beta$ is not a root of unity: since $-1=D_\ell(c,s)=\alpha^\ell+\beta^\ell$, we get that $-1=\beta^\ell(1+\rho^\ell)$. But $\alpha\beta=s$ and therefore, if $\rho$ is a root of unity, we have that $\beta^\ell(1+\rho^\ell)$ (hence $-1$) is divisible by a prime ideal which divides $s$ in the ring of integers of $\Q(\sqrt{1-4s})$ (change the roles of $\alpha$ and $\beta$ if needed), which is impossible. 

On the other hand, we have $\alpha^{2\ell}-\beta^{2\ell}=(\alpha^\ell+\beta^\ell)(\alpha^\ell-\beta^\ell)=-(\alpha^\ell-\beta^\ell)$ which means that $\alpha^{2\ell}-\beta^{2\ell}$ has no primitive divisors. Hence, by Theorem~\ref{t:BHV}, $2\ell\le30$ i.e. $\ell=3,5,7,11,13$. However, it is easy to see that $D_{\ell}(\pm1,s)\ne-1$ when $\ell=3,5,7,11,13$. 
\end{proof}

We need yet another lemma. We present here a short proof suggested by Daqing Wan in a personal communication with the author. 
 
\begin{lemma}\label{l:key}Let $g(x)=x^t-g_1x^{t-1}+g_2x^{t-2}-\dots+(-1)^{t}g_t\in\Z[x]$ be the minimal polynomial of $K_q(a)$. Then, for $k=1,\dots,t$, we have
\[
	g_k\equiv(-1)^k\binom{t}k\pmod p.
\]
\end{lemma}

\begin{proof}We recall, that every $\Q$-automorphism $\sigma_j$ of $\Q(\zeta_p)$ satisfies $\sigma_j(\zeta_p)=\zeta_p^j$, where $j\in\Z_p^*$ (see~\cite[Cor.~2,~p.195]{IR}), and it follows easily that $\sigma_j(K_q(a))=K_q(j^2a)$. 

Let $\lambda=1-\zeta_p$. Since $(p)=(\lambda)^{p-1}$ in $\Z[\zeta_p]$ (see~\cite[p.197]{IR}), we get by definition~\eqref{e:klsum} (with $n=1$) that $K_q(b)\equiv (q-1)\cdot1\equiv -1\pod\lambda$ for all $b\in\F_q^*$. Since $g_k$ is the $k$th elementary symmetric polynomial in the conjugates of $K_q(a)$, which are $K_q(j^2a)$ for some $j\in\F_p^*$, we now get that 
\[
	g_k\equiv\sum_{1\le i_1<\dots<i_k\le t}(-1)^k=(-1)^k\binom tk\pod\lambda. 
\]
Hence,
$g_k-(-1)^k\binom tk\in(\lambda)\cap\Z=p\Z$ i.e. $g_k\equiv(-1)^k\binom tk\pmod p$.  
\end{proof}

\medskip  

\noindent{\it Proof of Theorem~\ref{t:mainth}}. Let $\ell$ be a prime and assume $K_{q^\ell}(a)=-1$. It then follows from Theorem~\ref{t:Carlitz}, that the minimal polynomial $g(x)$ of $K_q(a)$ divides $D_{\ell}(x,q)+1$ in $\Z[x]$. Let $\bar g(x)\in\F_p[x]$ be the polynomial obtained by reducing the coefficients of $g(x)$ modulo $p$. If $\ell>2$, Theorem~\ref{t:irred_dickson} now implies that $g(x)=D_{\ell}(x,q)+1$. Now, by Lemma~\ref{l:key}, we get $(x+1)^\ell=\bar g(x)=x^{\ell}+1$. Hence $\ell=p$, which is impossible since the degree $\ell$ of $g(x)$ is at most $(p-1)/2$. If $\ell=2$, then either $\bar g(x)=x+1$ or $\bar g(x)=(x+1)^2$. The former case is impossible by Theorem~\ref{t:int}. In the latter case $(x+1)^2$ divides $x^2+1$ in $\F_p[x]$, which is impossible since $p>2$. Remark~\ref{r:reduction} now completes the proof.

\end{document}